\documentclass{amsart}

\newtheorem{theorem}{Theorem}[section]
\newtheorem*{theorem A}{Theorem A}
\newtheorem*{theorem B}{N\"olker's Theorem}

\newtheorem{corollary}{Corollary}[section]

\theoremstyle{remark}
\newtheorem{remark}{Remark}[section]
\theoremstyle{remark}

\theoremstyle{definition}

\numberwithin{equation}{section}
\def\({\left ( }
\def\){\right )}
\def\<{\left < }
\def\>{\right >}

\begin{document}

\noindent {\sc {International Electronic Journal of Geometry}}

\noindent {\sc \small Volume 5  No. 2 pp. 59--66 (2012) \copyright
IEJG}

\vspace{2cm}

\title{Some Characterizations of\\
almost Hermitian manifolds}

\author{Hakan Mete Ta\c stan}
\address{\.Istanbul University\\
Department of Mathematics\\
34134, Vezneciler, \.Istanbul-TURKEY}
\email{hakmete@istanbul.edu.tr}

\thanks{}


\date{}

\dedicatory{{\rm (Communicated by Grozio STALINOV)}}

\begin{abstract}
We give a condition for an almost constant-type manifold  to be a constant-type manifold, and holomorphic and $R$-invariant submanifolds of
almost Hermitian manifolds are studied. Generalizations of some results in
\cite{Ogi} are given.
\end{abstract}
\maketitle
\section{Introduction}
The theory of submanifolds of Riemannian or almost Hermitian
manifolds is an important topic in differential geometry. In an
almost Hermitian manifold, its complex structure $J$ transforms a
vector to another one which perpendicular to it. Perhaps this has
been the natural motivation to study submanifolds of an almost
Hermitian manifold, according to
the behavior of its tangent bundle under the action of the almost complex structure $J$ of the ambient manifold.\\

There are many classes of submanifolds in the literature. One of the
classes is \emph{holomorphic} (\emph{invariant}) submanifolds. In
this case the tangent space of the submanifold remains invariant
under the almost complex structure $J$. The other one is
$R$\emph{-invariant submanifolds}, where $R$ is the Riemannian
curvature tensor of the ambient manifold. In which case, $R_{XY}$
defines, at each point of the submanifold, a linear transformation
on the tangent space of the submanifold at the point and the tangent
space remains invariant under $R_{XY}$, where $X$ and $Y$ are
elements of the tangent space. The theory of holomorphic
submanifolds has been a
very active area whereas the theory of $R$-invariant submanifolds has not been so far.\\

This paper is organized as follows. In Section 2, the fundamental
definitions and notions of almost Hermitian manifolds are given. In
Section 3, we recall some basic formulas for  submanifolds of a
Riemannian manifold and the definition of holomorphic  submanifold
of an almost Hermitian manifold. In Section 4, we study almost constant-type manifolds
and give a characterization theorem. In section 5, after recalling the
definition of $R$-invariant submanifold we study holomorphic,
totally umbilical and $R$-invariant submanifolds of
almost Hermitian manifolds ($AH$-manifolds). In particular, we shall
improve Theorem 4.2 of \cite{Ogi} concerning Kaehlerian manifolds. In the last
section, the strongly $R$-invariant submanifolds are considered and
analogues of Theorem 3.4 and Theorem 3.5 of \cite{Ogi} concerning Riemannian manifolds are
given. A different characterization of almost constant-type manifolds is placed in this section.
\section{Preliminaries}
Let $M$ be an almost Hermitian manifold, that is, its tangent bundle
has an almost complex structure $J$ and a Riemannian metric $g$ such
that $g(JX,JY)=g(X,Y)$ for all $X,Y\in\chi(M)$, where $\chi(M)$ is
the Lie algebra of $C^{\infty}$ vector fields on $M$. Let $\nabla$
be the Riemannian connection on $M$, the Riemannian  curvature
tensor $R$ associated with $\nabla$ defined by
$R(X,Y)=\nabla_{[X,Y]}-[\nabla_{X},\nabla_{Y}].$ We denote
$g(R(X,Y)Z,W)$ by $R(X,Y,Z,W)$. Curvature identities are keys to understanding the geometry of almost Hermitian manifolds.
The following curvature identities are used in many studies, for
example, see \cite{Gan}.
\begin{enumerate}
\item $R(X,Y,Z,W)=R(X,Y,JZ,JW),$

\item  $R(X,Y,Z,W)=R(JX,JY,Z,W)+R(JX,Y,JZ,W)$

               $\quad\quad\quad\quad\quad\quad$$+R(JX,Y,Z,JW),$

\item $R(X,Y,Z,W)=R(JX,JY,JZ,JW)$.
\end{enumerate}
Let $AH_{i}$ denote the subclass of the class $AH$ of almost
Hermitian manifolds satisfying the curvature identity (i), i=1,2,3.
The following inclusion relations is well known in the literature.
$$AH_{1}\subset AH_{2}\subset AH_{3}\subset AH$$
An almost Hermitian manifold $M$ is called \emph{Kaehlerian} if
$\nabla_{X}J=0$ for all $X\in\chi(M)$. It is well known that a
Kaehlerian manifold is an $AH_{1}$-manifold (\cite{Yan}). Some authors call $AH_{1}$-manifold as a \emph{para-Kaehlerian} and call $AH_{3}$-manifold as an $RK$-\emph{manifold} (\cite{Van}).\\

By a \emph{plane section} we mean a  two-dimensional linear subspace
of a tangent space $T_{p}M$. A plane section $\sigma$ is said to
be \emph{holomorphic} (resp.\emph{anti-holomorphic} or \emph{totally
real}) if $J\sigma=\sigma$ (resp. $J\sigma\bot\sigma$). The
sectional curvature $K$ of $M$ determined by orthonormal vector
fields $X$ and $Y$ is given by $K(X,Y)=R(X,Y,X,Y).$The  sectional
curvature of $M$ restricted to a holomorphic (resp. an
anti-holomorphic) plane $\sigma$ is called \emph{holomorphic} (resp.
\emph{anti-holomorphic}) \emph{sectional curvature}. If the
holomorphic (resp. anti-holomorphic) sectional curvature at each
point $p\in M$, does not depend on $\sigma$, then $M$ is said to
be\emph{ pointwise constant holomorphic} (resp. \emph{pointwise
constant anti-holomorphic) sectional curvature}. A Riemannian
manifold of constant curvature is called a \emph{space form} or a
\emph{real space form}. Sometimes  a space form defined as a
complete connected Riemannian manifold of constant curvature  is
said to be \emph{elliptic}, \emph{hyperbolic} or \emph{flat} (or
\emph{locally Euclidean}) according as the sectional curvature is
positive, negative or zero. A Kaehlerian manifold of constant
holomorphic sectional curvature is called a \emph{complex space
form} (\cite{Kas,Yan}).
\section{Submanifolds of a Riemannian manifold}
Let $N$ be a submanifold of a Riemannian manifold $M$ with a
Riemannian metric $g.$ Let $\chi(N)$ and $\chi(M)$ the Lie algebras
of vector fields on $N$ and $M$ respectively, and
$\overline{\chi}(N)$ denote the Lie algebra of restrictions to $N$
of vector fields in $\chi(M)$. We then may write $\overline{\chi}(N)
=\chi(N)\oplus\chi(N)^{\bot}$, where $\chi(N)^{\bot}$ consists of
all vector fields perpendicular to $N$. The Gauss and Weingarten
formulas are given respectively by
$\nabla_{X}Y=\hat{\nabla}_{X}Y+B(X,Y)$  and
$\nabla_{X}\xi=-A_{\xi}X+\nabla_{X}^{\bot}\xi$ for all
$X,Y\in\chi(N)$ and $\xi\in\chi^{\bot}(N)$, where $\nabla,
\hat{\nabla}$ and $\nabla^{\bot}$ are respectively the Riemannian,
induced Riemannian and induced normal connection in $M, N$ and the
normal bundle $\chi^{\bot}(N)$ of $N$, and $B$ is the second
fundamental form related to shape operator $A$ corresponding to the
normal vector field $\xi$ by $g(B(X,Y),\xi)=g(A_{\xi}X,Y).$ We say
that $N$ is \emph{totally umbilical} submanifold in $M$ if for all
$X,Y\in\chi(N),$ we have
\begin{equation}
\label{e1}
\begin{array}{c}
B(X,Y)=g(X,Y)H
\end{array},
\end{equation}
where $H\in\chi^{\bot}(N)$ is the mean curvature vector field of $N$
in $M$. A vector field $\xi\in\chi^{\bot}(N)$ is said to be
 \emph{parallel} if $ \nabla^{\bot}_{X}\xi=0$ for each $X\in\chi(N)$. The Codazzi equation is given by
\begin{equation}
\label{e1}
\begin{array}{c}
(R(X,Y)Z)^{\bot}=(\nabla_{X}B)(Y,Z)-(\nabla_{Y}B)(X,Z)
\end{array}
\end{equation}
for all $X,Y,Z\in\chi(N),$ where $^{\bot}$ denotes the normal
component and the covariant derivative of $B,$ denoted by
$\nabla_{X}B$, is defined by
\begin{equation}
\label{e1}
\begin{array}{c}
(\nabla_{X}B)(Y,Z)=\nabla^{\bot}_{X}(B(Y,Z))-B(\hat{\nabla}_{X}Y,Z)-B(Y,\hat{\nabla}_{X}Z)
\end{array}
\end{equation}
for all $X,Y,Z\in\chi(N)$ \cite{Kas,Ogi,Yan}.\\

Now, let $(M,J,g)$ (or briefly $M$) be an almost Hermitian manifold
and $N$ be a Riemannian submanifold of $M$. If $J(T_{p}N)=T_{p}N$ at
each point $p\in N,$
 $T_{p}N$ being the tangent space over $N$ in $M,$ then $N$ is called a \emph{holomorphic} submanifold of $M.$ In this case we see that
$JT_{p}N^{\bot}= T_{p}N^{\bot}.$ The fundamental properties and
basic formulas of holomorphic submanifolds can be found in
\cite{Yan}.
\section{Almost Constant-type manifolds}
The following notion of constant type, is first defined by A. Gray for nearly Kaehlerian manifolds (\cite{Gray}),
and then by L. Vanhecke for almost Hermitian manifolds (\cite{Van}).\\

Let $M$ be an almost Hermitian manifold. Then $M$ is said
to be of \emph{constant type} at  $p\in M$ provided that for all $X\in T_{p}M$,
we have $\lambda(X,Y)=\lambda(X,Z)$ whenever the planes  $span\{X,Y\}$ and
$span\{X,Z\}$ are  anti-holomorphic and $g(Y,Y)=g(Z,Z)$,
where the function $\lambda$ is defined by
$\lambda(X,Y)=R(X,Y,X,Y)-R(X,Y,JX,JY)$. If this holds for all $p\in M$,
then we say that  $M$ has \emph{(pointwise) constant type}.
Finally, if for $X,Y\in\chi(M)$ with $g(X,Y)=g(JX,Y)=0,$ the value
$\lambda(X,Y)$ is constant whenever $g(X,X)=g(Y,Y)=1$,
then we say that $M$ has \emph{global constant type.}\\

In \cite{Rizza}, for an $AH$-manifold with dimension $2m\geq4$, G.B. Rizza gave an equivalent definition to the one above.
For some geometrical notions and for the notations, see \cite{Riza, Rizza}.\\

The manifold $M$ is said to be of \emph{constant type} $\alpha$ at $p$
 if and only if  we have
\begin{equation}
\label{e1}
\begin{array}{c}
K_{\sigma}-\chi_{\sigma J\sigma}=\alpha
\end{array},
\end{equation}
where $\sigma$ is any anti-holomorphic plane of $T_{p}M$ and $p\in M$.\\

In the same paper,  G.B. Rizza also generalized the notion of constant type as follows.\\

We  say that $M$ has \emph{constant type} $\alpha$, \emph{in a weak sense}, at $p$,
 if and only if  we have
\begin{equation}
\label{e1}
\begin{array}{c}
K_{\sigma}-2\chi_{\sigma J\sigma}+K_{J\sigma}=2\alpha
\end{array},
\end{equation}
where $\sigma$ is any anti-holomorphic plane of $T_{p}M$ and $p\in M$.

\begin{remark} If $\sigma$ is an anti-holomorphic plane of $T_{p}M$, then $J\sigma$ is also
an anti-holomorphic plane of $T_{p}M$, so (4.1) implies (4.2). In particular, if the curvature tensor
$R$ is $J$-invariant, that is, if $R$ satisfies the curvature identity (3), or the manifold $M$ belongs to
class $AH_{3}$, then we have $K_{\sigma}=K_{J\sigma}$. Thus (4.2) reduces to (4.1). (cf. \cite{Rizza}).
\end{remark}
An $AH$-manifold $M$ is said to be an \emph{almost constant-type manifold}
if and only if $M$ has constant type, in a weak sense at any point $p\in M$.
In particular, $M$ is a \emph{constant-type manifold} if the condition (4.1) is satisfied at any point $p\in M$, see \cite{Rizza}.\\

Now, we give a condition for an almost constant-type manifold  to be a constant-type manifold.
\begin{theorem} Let $M$ be an $AH$-manifold of almost constant-type $\alpha$. Then $M$ has constant type $\alpha$
if and only if $M$ belongs to class $AH_{3}$.
\end{theorem}
\begin{proof} Let $M$ be an $AH$-manifold of constant type $\alpha$, in a weak sense, at $p\in M.$
From the equation (5) in the proof of Theorem 1(\cite{Rizza}), we have
\begin{equation}
\label{e1}
\begin{array}{c}
\quad\quad\quad R(X,Y,X,Y)-2R(X,Y,JX,JY)+R(JX,JY,JX,JY)\\
=2\alpha\{g(X,X)g(Y,Y)-(g(X,Y))^{2}-(g(JX,Y))^{2}\}
\end{array}
\end{equation}
for any vectors $X,Y\in T_{p}M$.
Now, suppose that $M$ has constant type $\alpha$, at $p\in M.$
By a similar method given in  the proof of Theorem 1(\cite{Rizza}), we obtain
\begin{equation}
\label{e1}
\begin{array}{c}
R(X,Y,X,Y)-R(X,Y,JX,JY)\\
\quad\quad\quad\quad\quad\quad\quad=\alpha\{g(X,X)g(Y,Y)-(g(X,Y))^{2}-(g(JX,Y))^{2}\}
\end{array}
\end{equation}
for any vectors $X,Y\in T_{p}M$. ( The equation (4.4) is also known in another form, see \cite{Kiri}).
If we multiply the equation (4.4) by 2  and then subtract it from the equation (4.3), we get
\begin{equation}
\label{e1}
\begin{array}{c}
R(X,Y,X,Y)=R(JX,JY,JX,JY)
\end{array}.
\end{equation}
Now, let $ T:(T_{p}(M))^{4}\rightarrow\mathbb{R}$ be a four-linear mapping given by
$$T(X,Y,Z,W)=R(X,Y,Z,W)-R(JX,JY,JZ,JW)$$
for all $X,Y,Z,W\in T_{p}(M).$ Using the well-known properties of $R$ and (4.5),
it is not difficult to see that $T$ satisfies the conditions of Theorem 2(\cite{Gan}),
so it follows that  $T=0.$ This means that the curvature tensor $R$ of
$M$ satisfies the curvature identity (3); that is, $M$ belongs to class $AH_{3}$,
in which case, by virtue of  Remark 1(\cite{Rizza}), $M$ also belongs to the class $AH_{2}$.
On the other hand, by Remark 4.1, the converse of Theorem 4.1. is also true.
\end{proof}
\section{$R$-invariant submanifolds}
Let $(M,g)$ be a Riemannian manifold and $N$ be a submanifold of
$M$. Then $R(X,Y)=R_{XY}$ defines, at each point $p$ of $N$, a
linear transformation on the tangent space $T_{p}N$ at the point,
where $R$ is the Riemannian curvature tensor of $M$. If for any
$X,Y\in T_{p}N,$ the relation $R_{XY}(T_{p}N)\subset T_{p}N$ is
satisfied, then $N$ is called an $R$-\emph{invariant} submanifold of
$M$. For the equivalent definitions of $R$-invariant submanifolds,
see \cite{Ogi}.
\begin{theorem}If every holomorphic submanifold of a connected $AH$-manifold $M$
of dimension $2m\geq6$ is $R$-invariant, then\\
\textbf{a)} $M$ is a real space form or a complex space form.\\
\textbf{b)} $M$ is an $AH_{2}$-manifold.\\
\textbf{c)} The curvature tensor $R$ of $M$ has the form
\begin{equation}
\label{e7}
\begin{array}{c}
R(X,Y,Z,W)=\alpha R_{1}(X,Y,Z,W)+\beta R_{2}(X,Y,Z,W)
\end{array}
\end{equation}
where $\alpha,\beta$ are constants,
$R_{1}(X,Y,Z,W)=g(X,W)g(Y,Z)-g(X,Z)g(Y,W)$ and
$R_{2}(X,Y,Z,W)=g(JX,W)g(JY,Z)-g(JX,Z)g(JY,W)-2g(JX,Y)g(JZ,W)$ for
all  $X,Y,Z,W\in T_{p}M$ and $p\in M.$\\
Moreover, the converse of \textbf{c)} is also true.
\end{theorem}
\begin{proof} If for any point $p$ in $M$ of a connected $AH$-manifold $M$
of dimension $2m\geq6$ and for any holomorphic subspace $\sigma$ of $T_{p}M$
there exists an $R$-invariant holomorphic submanifold $N$ of $M$ through $p$
such that $T_{p}N$ contains $\sigma$, then we have
\begin{equation}
\label{e7}
\begin{array}{c}
(R(X,JX)JX)^{\bot}=0
\end{array}
\end{equation}
and
\begin{equation}
\label{e7}
\begin{array}{c}
(R(X,Y)Y)^{\bot}=0
\end{array}
\end{equation}
for all orthonormal vectors $X,Y\in\sigma $ with  $g(X,JY)=0$ and
$p\in N$. For any vector $\xi$ normal to $N$ at $p,$ from (5.2) and
(5.3), we get
\begin{equation}
\label{e7}
\begin{array}{c}
R(X,JX,JX,\xi)=0
\end{array}
\end{equation}
and
\begin{equation}
\label{e7}
\begin{array}{c}
R(X,Y,Y,\xi)=0
\end{array}.
\end{equation}
From (5.4) and Lemma 1(\cite{Kas}), we conclude
that $M$ has pointwise constant holomorphic sectional curvature $\mu$ and from (5.5)
and Lemma 4(\cite{Kas}), we find that $M$ has pointwise constant anti-holomorphic sectional
curvature $\nu.$  In this case, the assertion \textbf{a)} follows
from Theorem C(\cite{Kas}) and assertions \textbf{b)} and
\textbf{c)} follow from Theorem 3(\cite{Gan}).\\
Now we prove the converse of \textbf{c)}. Let the curvature tensor
$R$ of $M$ satisfy (5.1), then after some calculation in (5.1), we
have
\begin{equation}
\label{e7}
\begin{array}{c}
R(X,Y)Z=\alpha\{g(Y,Z)X-g(X,Z)Y\}\quad\quad\quad\quad\quad\quad\\
\quad\quad\quad\quad\quad\quad\quad\quad\quad\quad+\beta\{g(JY,Z)JX-g(JX,Z)JY-2g(JX,Y)JZ\}
\end{array}
\end{equation}
for all  $X,Y,Z,W\in T_{p}M$ and $p\in M.$ In that case, if $N$ is
any holomorphic submanifold of the manifold $M,$ then from (5.6), we easily see that $R(X,Y)Z\in T_{p}N$ for all
$X,Y,Z\in T_{p}N$ and $p\in N$, that is, $N$ is $R$-invariant.
\end{proof}
\begin{remark} The assertions \textbf{a)} and \textbf{c)} of Theorem 5.1. are equivalent whenever
the ambient manifold $M$  is of class $AH_{3}$ because of Theorem 12.7(\cite{Tri}).
\end{remark}
\begin{theorem} Let $N$ be any connected totally umbilical submanifold of an $AH$-manifold $M$.
Then $N$ is an $R$-invariant submanifold if and only if it has
parallel mean curvature vector field.
\end{theorem}
\begin{proof} Let $N$ be any connected totally umbilical submanifold of an $AH$-manifold
$M$. With the help of (3.1) and (3.3), from the Codazzi equation
(3.2), we have
\begin{equation}
\label{e7}
\begin{array}{c}
(R(X,Y)Z)^{\bot}=g(Y,Z)\nabla^{\bot}_{X}H-g(X,Z)\nabla^{\bot}_{Y}H
\end{array}
\end{equation}
for all $X,Y,Z\in\chi(N),$ where $H$ is mean curvature vector field
of $N$ in $M$. Let $N$ be $R$-invariant, then from (5.7), we have
$(R(X,Y)Z)^{\bot}=0$, in which case, if we choose $Y=Z$ with
$g(X,Y)=0$ in (5.7), we obtain $\nabla^{\bot}_{X}H=0$, which says
that $H$ is parallel. On the other hand, if $H$ is parallel, from
(5.7), we easily see that $(R(X,Y)Z)^{\bot}=0$, which means that $N$
is an $R$-invariant submanifold.
\end{proof}
We remark that  Theorem 5.1 is a generalization  Theorem 4.2(\cite{Ogi}) concerning Kaehlerian manifolds  and analogue of Theorem 2.2(\cite{Ogi}) concerning Riemannian manifolds.
\section{Strongly $R$-invariant submanifolds}
Let $(M,g)$ be a Riemannian manifold and $N$ be a submanifold of
$M$. If for any $X,Y\in T_{p}M,$ the relation $R_{XY}(T_{p}N)\subset
T_{p}N$ is satisfied, then $N$ is called a \emph{strongly}
$R$-\emph{invariant} submanifold of $M$. It is easy to see that a
strongly $R$-invariant submanifold is an $R$-invariant submanifold.
For the equivalent definitions of strongly $R$-invariant
submanifolds, see \cite{Ogi}.
\begin{theorem} There exists no holomorphic strongly $R$-invariant submanifold of an $AH$-manifold
$M$ with pointwise non-zero constant holomorphic sectional curvature
$\mu.$
\end{theorem}
\begin{proof} In  \cite{Riza}, G.B. Rizza, for an $AH$-manifold $M$ with pointwise
holomorphic constant sectional curvature $\mu$, proved the
following fundamental identity:
\begin{equation}
\label{e7}
\begin{array}{c}
3\{R(X,Y,Z,W)+R(X,Y,JZ,JW)+R(JX,JY,Z,W)\\
\quad\quad\quad+R(JX,JY,JZ,JW)\}-2\{R(X,W,JY,JZ)+R(JX,JW,Y,Z)\\
\quad-R(X,Z,JY,JW)-R(JX,JZ,Y,W)\}-\{R(X,JY,JZ,W)\\
\quad+R(JX,Y,Z,JW)+R(X,JY,Z,JW)+R(JX,Y,JZ,W)\}\\
\quad\quad=4\mu\{g(X,Z)g(Y,W)-g(X,W)g(Y,Z)+g(X,JZ)g(Y,JW)\\
-g(X,JW)g(Y,JZ)+2g(X,JY)g(Z,JW)\}\quad\quad\quad\quad
\end{array}
\end{equation}
for all $X,Y,Z,W\in T_{p}M$ and $p\in M$. After some calculations in (6.1), we conclude that
\begin{equation}
\label{e7}
\begin{array}{c}
3\{R(X,Y)Z-J(R(X,Y)JZ)+R(JX,JY)Z\\
\quad\quad\quad\quad-J(R(JX,JY)JZ)\}-2\{R(JY,JZ)X-J(R(Y,Z)JX)\\
\quad\quad+J(R(X,Z)JY)-R(JX,JZ)Y\}-\{R(X,JY)JZ\\
\quad\quad-J(R(JX,Y)Z)-J(R(X,JY)Z)+R(JX,Y)JZ\}\\
=4\mu\{g(X,Z)Y-g(Y,Z)X-g(X,JZ)JY\quad\quad\\
+g(Y,JZ)JX-2g(X,JY)JZ\}\quad\quad\quad\quad\quad
\end{array}
\end{equation}
for all $X,Y,Z\in T_{p}M$ and $p\in M.$\\
Now, assume that $N$ is a holomorphic strongly $R$-invariant submanifold of an $AH$-manifold
$M$ with pointwise non-zero constant holomorphic sectional curvature
$\mu.$ Then from (6.2) we have
\begin{equation}
\label{e7}
\begin{array}{c}
3R(X,\xi)X+5R(JX,J\xi)X-R(X,J\xi)JX-R(JX,\xi)JX\quad\quad\quad\quad\quad\\
-5J(R(X,\xi)JX)-3J(R(JX,J\xi)X)+J(R(X,J\xi)X)+J(R(JX,\xi)X)\\
=4\mu g(X,X)\xi\quad\quad\quad\quad\quad\quad\quad\quad\quad\quad\quad\quad\quad\quad\quad\quad\quad\quad\quad\quad\quad\quad\quad\quad\quad
\end{array}
\end{equation}
for all $X\in T_{p}N, p\in N$ and $\xi\in
T_{p}N^{\bot}.$ From Proposition 3.3(\cite{Ogi}) and (6.3), we find
$\mu=0.$
\end{proof}
Since a strongly $R$-invariant submanifold is necessarily
$R$-invariant submanifold, from Theorem 5.1 and Theorem 6.1, we have
the following result.
\begin{corollary} If every holomorphic submanifold of a connected $AH$-manifold $M$ of dimension $2m\geq6,$
is a strongly $R$-invariant submanifold, then $M$ is flat.
\end{corollary}
We note that Theorem 6.1 and Corollary 6.1 may
be considered as analogues of Theorem 3.4(\cite{Ogi})
and Theorem 3.5(\cite{Ogi}) concerning Riemannian manifolds respectively,
and Theorem 6.1 is a generalizations
of Theorem 4.3(\cite{Ogi}) concerning Kaehlerian manifolds.
We end this paper, giving a different characterization of almost constant-type $AH$-manifolds.
\begin{theorem} There exists no holomorphic strongly $R$-invariant submanifold of an $AH$-manifold
$M$ of non-zero constant type $\alpha$, in a weak sense, at $p\in M$.
\end{theorem}
\begin{proof}Assume that $N$ is a holomorphic strongly $R$-invariant submanifold of an $AH$-manifold
$M$ of non-zero constant type $\alpha$, in a weak sense, at $p\in M$. Then from (4.3),
for $X\in T_{p}N$ and $\xi\in T_{p}N^{\bot}$ with $g(X,X)=g(\xi,\xi)=1$, we have
\begin{equation}
\label{e1}
\begin{array}{c}
R(X,\xi,X,\xi)-2R(X,\xi,JX,J\xi)+R(JX,J\xi,JX,J\xi)=2\alpha.
\end{array}
\end{equation}
From Proposition 1.4(\cite{Ogi}) and (6.4), it follows that
$\alpha=0.$ This is a contradiction.
\end{proof}
\subsection*{Acknowledgements}
The author thanks the referee(s) for their comments and suggestions
to improve this paper.


\begin{thebibliography}{99}
\bibitem {Gan}G.T. Gancev, Almost Hermitian manifolds similar to complex space forms, C.R. Acad. Bulgare Sci.
32(1979), 1179-1182.
\bibitem {Gray}A. Gray, Nearly K$\ddot{a}$hler manifolds, J. Differential Geometry
4(1970), 283-309.
\bibitem {Kas}O.T. Kassabov, On the axiom of planes and the axiom of spheres in the almost Hermitian geometry, Serdica
8(1982), no.1,  109-114.
\bibitem {Kiri} V.F. Kirichenko and I.V. Tret'yakova, On the constant-type manifolds of almost Hermitian manifolds, Mathematical Notes,
68(2000), no.5,  569-575.
\bibitem {Ogi}K. Ogiue, On invariant immersions, Ann. Math. Pura Appl.
4(1968), 387-397.
\bibitem {Riza}G.B. Rizza, On almost Hermitian manifolds with constant holomorphic sectional curvature at a point, Tensor, N.S., 50(1991), 79-89.
\bibitem {Rizza}G.B. Rizza, On almost constant-type manifolds, Journal of Geometry, 48(1993), 174-183.
\bibitem {Tri}F. Tricerri and L. Vanhecke,  Curvature tensors on almost Hermitian manifolds, Trans. Amer. Math. Soc.
267(1981), no.2, 365-398.
\bibitem {Van}L. Vanhecke,  Almost Hermitian manifolds with $J$-invariant Riemann curvature tensor, Rend. Sem. Mat. Univers. Politech. Torino
34(1975-76), 487-498.
\bibitem {Yan}K. Yano  and M. Kon, Structures on Manifolds, World Scientific, Singapore, 1984.
\end{thebibliography}
\end{document}